\documentclass[leqno,10pt]{article} 
\usepackage{latexsym}
\usepackage{amssymb}
\newcommand{\ds}{\displaystyle}
\newcommand{\barF}{\overline{F}}
\newcommand{\barE}{\overline{E}}
\newcommand{\barP}{\overline{P}}
\newcommand{\spazio}{\vspace{1mm}}  
\newcommand{\qed}{\hfill\rule{2mm}{2mm}}      
\newtheorem{proposition}{Proposition}[section]
\newenvironment{proof}{\begin{trivlist}
\item[\hspace{\labelsep}{\bf\noindent Proof. }]}
{\qed\end{trivlist}}
\date{\empty}
\title{\large
{\bf COMPETING RISKS WITHIN SHOCK MODELS}\footnote{Paper appeared in 
{\em Scientiae Mathematicae Japonicae} 67 (2008), No. 2, 125--135.  
\newline
{\em 2000 Mathematics Subject Classification:} 62N05, 60K10.  
\newline
{\em Key words and phrases.} Poisson shock models, competing risks model, failure distributions.  
}}
\author{\sc Antonio Di Crescenzo and Maria Longobardi}
\begin{document}
\maketitle
\begin{abstract}
We consider a competing risks model, in which system failures are due to one 
out of two mutually exclusive causes, formulated within the framework of shock 
models driven by bivariate Poisson process. We obtain the failure densities 
and the survival functions as well as other related quantities under three 
different schemes. Namely, system failures are assumed to occur at the first 
instant in which a random constant threshold is reached by (a) the sum of 
received shocks, (b) the minimum of shocks, (c) the maximum of shocks. 
\end{abstract}
%
\section{Introduction}
The classical competing risks model deals with failure times subject to multiple 
causes of failure. This model is of specific interest in various applied fields such 
as survival analysis and reliability theory. Indeed, it is appropriate for instance 
for describing failures of organisms or devices in the presence of different 
types of risks. Usually this model involves an observable pair $(T,\delta)$, 
where $T$ is the time of failure and $\delta$ describes cause or type of 
failure. General properties of this model can be found in the literature 
(see Bedford and Cooke \cite{BeCo01} and Crowder \cite{Cr01}). For a list of 
references and recent results on competing risks model see Di Crescenzo 
and Longobardi \cite{DiLo2006}, and other contributions in \cite{JSPI2006}. 
\par
In this paper we introduce a formulation of the competing risks model within 
the framework of stochastic shock models. These have been introduced and 
studied with the aim of describing systems subject to shocks occurring 
randomly such as events in counting processes. About 40 years ago 
A.W.\ Marshall and I.\ Olkin began to investigate probabilistic shock models. 
In particular, in \cite{MaOl66} and \cite{MaOl67} they obtained certain 
bivariate exponential distributions from fatal and non-fatal shock models. 
Later, by means of the total positivity theory, Esary {\em et al.}\ \cite{EsMaPr73} 
introduced a classical preservation problem in shock models, i.e.\ to show that 
certain properties of the discrete distribution of the number of shocks that leads 
a system to failure are reflected in corresponding properties of the continuous 
distribution of the system lifetime. This investigation line has been followed  
by numerous authors including, for instance, El-Neweihi {\em et al.}\ 
\cite{ENPrSe83}. We also recall the contributions by Gottlieb \cite{Go80} 
and by Ghosh and Ebrahimi \cite{GoEb82}, concerning 
the increasing failure rate property of the life distribution in shock models. 
Furthermore, preservations results for partial stochastic orderings and for 
classes of failure distributions in various types of shock models have been 
obtained by Klefsj\"o \cite{Kl81}, Fagiuoli and Pellerey \cite{FaPe93}, 
\cite{FaPe94a}, \cite{FaPe94b}, Kebir \cite{Ke94}, Kochar \cite{Ko90}, 
Pellerey \cite{Pe93}, Singh and Jain \cite{SiJa89}. Moreover, cumulative 
shock models in univariate and multivariate cases have been considered 
by Gut \cite{Gu90}, Kijima and Nakagawa \cite{KiNa91}, P\'erez-Oc\'on and 
G\'amiz-P\'erez \cite{PeGa95}, \cite{PeGa96}, Shaked and Shanthikumar \cite{ShSh91}. 
\par
Recently, new general classes of shock models of interest in reliability theory 
and survival analysis have been proposed, in which failure is due to the 
competing causes of degradation and trauma (see Lehmann \cite{Le06}, where 
the survival function of the failure time has been computed). This research 
direction is motivated by the circumstance that failure mechanisms can 
often be ascribed to an underlying degradation process and stochastically 
changing covariates. The present paper falls within such a line of investigation, 
aiming to include the presence of competing risks into the classical shock 
models scheme. 
\par
In our setting we consider systems characterized by two types of shocks and assume 
that failures are due to a single cause, with shocks arriving according to a pair 
of independent counting processes $N_1(t)$ and $N_2(t)$. We study in detail the 
quantities of interest under three different failure schemes. In each scheme, 
failures are assumed to occur at the first instant in which one of the 
following relations holds: 
\par
(a) \ $N_1(t) + N_2(t) = M$, 
\par
(b) \ $\min\{N_1(t), N_2(t)\} = M$, 
\par
(c) \ $\max\{N_1(t), N_2(t)\} = M$, \\
where $M$ is a random counting number. We examine various cases arising 
from suitable choices of the distribution of $M$. 
\par
In Section 2 we describe the general setting of the model by introducing the 
relevant hazard rates and obtain failure densities, survival function, 
failure probability and the failure time moments conditional on cause or 
type of failure $\delta$. In Section 3 we study this model for each of 
the above schemes under the assumption that $N_1(t)$ and $N_2(t)$ are 
homogeneous Poisson processes. When failures are due to sum of shocks, 
failure time $T$ and type of failure $\delta$ are shown to be independent, 
and an approximation of the survival function in the presence of an 
arbitrary (large) number $h$ of types of shocks is given. Finally, in 
Section 4 we suggest how to extend this model to the case of two 
non-exclusive types of shocks. 
\section{A model with two kinds of shocks}
Recently, various investigations have been oriented towards multidimensional shock 
models. See, for instance, Ohi and Nishida \cite{OhNi78}, \cite{OhNi79}, and 
Balu and Sabnis \cite{BaSa97} for various results on bivariate shock models, 
including properties of the joint survival probability and preservation of 
reliability structures, and Wong \cite{Wo97} and Pellerey \cite{Pe99} for 
preservation of stochastic orders in multivariate shock models with underlying 
counting processes. A multidimensional shock model driven by counting process
is considered in Li and Xu \cite{LiXu01}, where each shock may simultaneously 
destroy a subset of the components of a system consisting of several components. 
We shall now introduce a new kind of shock model characterized by two types 
of shocks, the extension to a higher number of types of shocks being 
straightforward. This has been inspired by the shock model with aftereffects 
treated in Marshall and Olkin \cite{MaOl67} and in Ghurye and Marshall \cite{GuMa80}. 
\par
Let $T$ be an absolutely continuous non-negative random variable describing the 
random failure time of a system or of a living organism. We set $\delta=i$ if 
the failure occurs due to a shock of type $i$, for $i=1,2$. Let $N(t)$ denote 
the total number of shocks occurring in $[0,t],$ with $t\geq 0$. We have 
$$
 N(t) = N_1(t) + N_2(t), \qquad t\geq 0,
$$
where $N_i(t)$ is the counting process representing the number of shocks 
of type $i$ occurring in $[0,t]$, $i=1,2$. 
\par 
The state-space $\mathbb{N}^2_0$ of $(N_1(t),N_2(t))$ is partitioned 
into non-empty subsets $S_k$, $k=0,1,2,\dots$, that will be called {\em failure 
sets}, $S_0$ including the numbers of non-fatal shocks. In other terms, if 
$(N_1(t),N_2(t))\in S_0$ then the shocks that arrived up to time $t$ do not cause 
a failure. Let $M$ be the integer-valued random variable that represents the 
index of which failure set $S_k$, $k=1,2,\dots$, containes the numbers of fatal 
shocks. This means that if $M=k$ then the system fails at the first instant $t>0$ 
in which $(N_1(t),N_2(t))\in S_k$. The probability distribution and the survival 
probability of $M$ will be respectively denoted by 
\begin{equation}
 p_k=P(M=k), \qquad k=1,2,\dots,
 \label{equation:4}
\end{equation}
and
\begin{equation}
 \barP_k=P(M>k), \qquad k=0,1,2,\dots.
 \label{equation:5}
\end{equation}
For $k=1,2,\dots$ we now define 
\begin{equation}
\begin{array}{l}
 \tilde S_k^{(1)}= \{(x_1,x_2)\in \mathbb{N}^2_0-S_k:
 (x_1+1,x_2)\in S_k\}, \\
 \\
 \tilde S_k^{(2)}= \{(x_1,x_2)\in \mathbb{N}^2_0-S_k:
 (x_1,x_2+1)\in S_k\}.
\end{array}
\label{equation:21}
\end{equation}
These will be called {\em risky sets}, because $\tilde S_k^{(i)}$, $i=1,2$, 
containes all states of $\mathbb{N}^2_0$ that lead to a failure at time $t$ 
at the next occurrence of a shock of type $i$, given that $M=k$. Note that 
definitions (\ref{equation:21}) assume that failures are due to a single cause. 
Moreover, note that in general $\tilde S_k^{(i)}$ is different from $S_{k-1}$. 
\par 
Denoting by $f_T(t)$, $t\geq 0$, the probability density function of the failure 
time $T$, we have 
\begin{equation}
 f_T(t)=f_1(t)+f_2(t), \qquad t\geq 0,
 \label{equation:13}
\end{equation}
where $f_i(t)$ is the sub-density defined by 
$$ 
 f_i(t)=\frac{\rm d}{{\rm d}t}P\{T\leq t,\,\delta=i\}, 
 \qquad t\geq 0, \quad i=1,2.
$$
Recalling that $\delta=i$ if the failure occurs due to a shock of type $i$, 
we have 
\begin{equation}
 P(\delta=i)=\int_0^{\infty} f_i(t)\,{\rm d}t, 
 \qquad i=1,2.
 \label{equation:12}
\end{equation}
\par
In order to express the sub-densities $f_i(t)$, $i=1,2$, in terms of the joint 
probability distribution of $(N_1(t),N_2(t))$ we now introduce the hazard rates 
\begin{equation}
 \begin{array}{l}
 r_1(x_1,x_2; t)=\ds\lim_{\tau\to 0^+} \frac{1}{\tau}
 P\{N_1(t+\tau)= x_1+1, N_2(t+\tau)= x_2| N_1(t)=x_1, N_2(t)=x_2\}, \\
 \\   
 r_2(x_1,x_2; t)=\ds\lim_{\tau\to 0^+} \frac{1}{\tau}
 P\{N_1(t+\tau)= x_1, N_2(t+\tau)= x_2+1| N_1(t)=x_1, N_2(t)=x_2\},
 \end{array}
 \label{equation:7}
\end{equation}
with $(x_1,x_2)\in\mathbb{N}^2_0$ and $t\geq 0$. Given that $x_1$ shocks of type 
1 and $x_2$ shocks of type 2 occurred in $[0,t]$, $r_i(x_1,x_2; t)$ gives the 
intensity of the occurrence of a shock of type $i$ immediately after $t$, with $i=1,2$. 
We note that $r_1(x_1,x_2; t)+r_2(x_1,x_2; t)$ is the hazard rate of a shock of 
type 1 or 2.  
\par
From the above assumptions it follows that the system fails around time $t$ due to 
a shock of type $i$ if the two-dimensional counting process $(N_1(t),N_2(t))$ 
takes values in $\tilde S_k^{(i)}$ at time $t$ {\em and} a shock of type $i$ occurs 
immediately after. Hence, conditioning on $M$ and recalling (\ref{equation:4}), 
(\ref{equation:21}) and (\ref{equation:7}), for $t\geq 0$ and $i=1,2$, 
failure densities can be expressed as 
\begin{equation}
 f_i(t)=\sum_{k=1}^{+\infty} p_k \sum_{(x_1,x_2)\in \tilde S_k^{(i)}}
 P\{N_1(t)=x_1,N_2(t)=x_2\}\,r_i(x_1,x_2;t).
 \label{equation:6}
\end{equation}
A relation similar to (\ref{equation:6}) holds for the survival function 
of $T$, denoted by  
$$
 \barF_T(t)=P\{T>t\}, \qquad t\geq 0. 
$$
Indeed, conditioning on $(N_1(t),N_2(t))$ and recalling (\ref{equation:5}), we obtain
\begin{equation}
 \barF_T(t)=\sum_{k=0}^{+\infty} \barP_k \sum_{(x_1,x_2)\in S_k} 
 P\{N_1(t)=x_1,N_2(t)=x_2\},
 \qquad t\geq 0,
 \label{equation:9}
\end{equation}
where $\barP_0=1$. Other quantities of interest are the moments of the failure 
time conditional on the cause of failure:
\begin{equation}
 E(T^s\,|\,\delta=i)={1\over P(\delta=i)}\int_{0}^{+\infty} t^s\,f_i(t)\,{\rm d}t,
 \qquad i=1,2. 
 \label{equation:23}
\end{equation}
This can be evaluated making use of (\ref{equation:12}) and (\ref{equation:6}). 
\section{Shocks driven by a Poisson process}
In this Section we assume that the two kinds of shocks affect the system 
according to independent homogeneous Poisson processes $N_1(t)$ and $N_2(t)$. 
Therefore, for $x_1,x_2=0,1,2,\dots$ and $t\geq 0$ one has: 
\begin{equation}
 P\{N_1(t)=x_1,N_2(t)=x_2\}=\frac{e^{-\lambda_1 t}(\lambda_1 t)^{x_1}}{x_1!}
 \frac{e^{-\lambda_2 t}(\lambda_2 t)^{x_2}}{x_2!},
 \label{equation:8}
\end{equation}
with $\lambda_1>0$ and $\lambda_2>0$. Under assumption (\ref{equation:8}), from 
(\ref{equation:7}) we now have  
$$
 r_i(x_1,x_2; t)=\lambda_i, 
 \qquad (x_1,x_2)\in \mathbb{N}^2_0, 
 \quad t\geq 0, \quad i=1,2.
$$
Hence, from (\ref{equation:6}) and (\ref{equation:9}) we obtain:
\begin{equation}
 f_i(t)=\sum_{k=1}^{+\infty} p_k \sum_{(x_1,x_2)\in \tilde S_k^{(i)}}
 \frac{e^{-\lambda_1 t}(\lambda_1 t)^{x_1}}{x_1!}
 \frac{e^{-\lambda_2 t}(\lambda_2 t)^{x_2}}{x_2!}\lambda_i, 
 \qquad t\geq 0, \quad i=1,2
 \label{equation:10}
\end{equation}
and
\begin{equation}
 \barF_T(t)=\sum_{k=0}^{+\infty} \barP_k \sum_{(x_1,x_2)\in S_k}
 \frac{e^{-\lambda_1 t}(\lambda_1 t)^{x_1}}{x_1!}
 \frac{e^{-\lambda_2 t}(\lambda_2 t)^{x_2}}{x_2!}
 \qquad t\geq 0. 
 \label{equation:11}
\end{equation}
Recalling (\ref{equation:12}) and (\ref{equation:10}), it is not difficult to see that 
\begin{equation}
 P(\delta=i) =\pi_i \sum_{k=1}^{+\infty} p_k \sum_{(x_1,x_2)\in \tilde S_k^{(i)}}
 {(x_1+x_2)!\over x_1!\,x_2!}\,\pi_1^{x_1} \,\pi_2^{x_2}, 
 \qquad i=1,2,
 \label{equation:24}
\end{equation}
where we have set $\pi_i=\lambda_i/(\lambda_1+\lambda_2)$, $i=1,2$. 
Similarly, from (\ref{equation:23}) we obtain the conditional moments  
\begin{equation}
 E(T^s\,|\,\delta=i)={1\over (\lambda_1+\lambda_2)^s}\,
 \ds\frac{\ds\sum_{k=1}^{+\infty} p_k \sum_{(x_1,x_2)\in \tilde S_k^{(i)}}
 \ds{(x_1+x_2+s)!\over x_1!\,x_2!}\,\pi_1^{x_1} \,\pi_2^{x_2}}
 {\ds\sum_{k=1}^{+\infty} p_k \sum_{(x_1,x_2)\in \tilde S_k^{(i)}}
 \ds{(x_1+x_2)!\over x_1!\,x_2!}\,\pi_1^{x_1} \,\pi_2^{x_2}}, 
 \qquad i=1,2. 
 \label{equation:25}
\end{equation}
The above expressions show that the structures of the failure sets and 
of the risky sets are essential to specify the nature of the shock model. 
Hereafter we consider three special cases of interest, in which these sets 
are chosen according to typical models of reliability theory, as sketched 
in Figure \ref{figure:1} for $S_k$. 
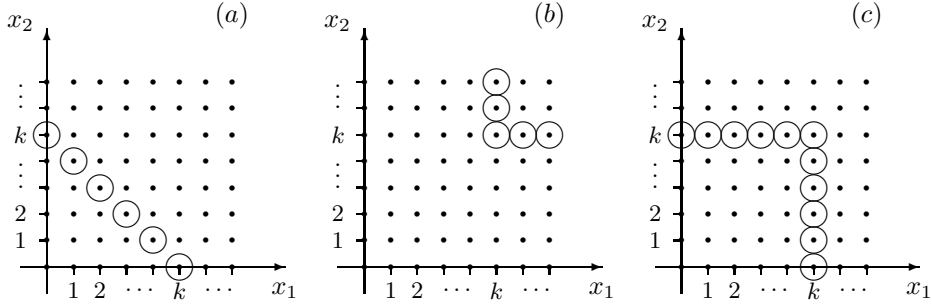
\begin{figure}
\begin{center}
\begin{picture}(341,146) 
\put(0,35){\vector(1,0){100}} 
\put(10,25){\vector(0,1){100}} 
\put(120,35){\vector(1,0){100}} 
\put(130,25){\vector(0,1){100}} 
\put(240,35){\vector(1,0){100}} 
\put(250,25){\vector(0,1){100}} 
\put(20,32){\line(0,1){3}} 
\put(30,32){\line(0,1){3}} 
\put(40,32){\line(0,1){3}} 
\put(50,32){\line(0,1){3}} 
\put(60,32){\line(0,1){3}} 
\put(70,32){\line(0,1){3}} 
\put(80,32){\line(0,1){3}} 
\put(140,32){\line(0,1){3}} 
\put(150,32){\line(0,1){3}} 
\put(160,32){\line(0,1){3}} 
\put(170,32){\line(0,1){3}} 
\put(180,32){\line(0,1){3}} 
\put(190,32){\line(0,1){3}} 
\put(200,32){\line(0,1){3}} 
\put(260,32){\line(0,1){3}} 
\put(270,32){\line(0,1){3}} 
\put(280,32){\line(0,1){3}} 
\put(290,32){\line(0,1){3}} 
\put(300,32){\line(0,1){3}} 
\put(310,32){\line(0,1){3}} 
\put(320,32){\line(0,1){3}} 
\put(7,45){\line(1,0){3}} 
\put(7,55){\line(1,0){3}} 
\put(7,65){\line(1,0){3}} 
\put(7,75){\line(1,0){3}} 
\put(7,85){\line(1,0){3}} 
\put(7,95){\line(1,0){3}} 
\put(7,105){\line(1,0){3}} 
\put(127,45){\line(1,0){3}} 
\put(127,55){\line(1,0){3}} 
\put(127,65){\line(1,0){3}} 
\put(127,75){\line(1,0){3}} 
\put(127,85){\line(1,0){3}} 
\put(127,95){\line(1,0){3}} 
\put(127,105){\line(1,0){3}} 
\put(247,45){\line(1,0){3}} 
\put(247,55){\line(1,0){3}} 
\put(247,65){\line(1,0){3}} 
\put(247,75){\line(1,0){3}} 
\put(247,85){\line(1,0){3}} 
\put(247,95){\line(1,0){3}} 
\put(247,105){\line(1,0){3}} 
\put(10,14){\makebox(20,15)[t]{\small 1}} 
\put(20,14){\makebox(20,15)[t]{\small 2}} 
\put(35,12){\makebox(20,15)[t]{\small $\ldots$}} 
\put(45,14){\makebox(30,15)[t]{\small $k$}} 
\put(65,12){\makebox(20,15)[t]{\small $\ldots$}} 
\put(130,14){\makebox(20,15)[t]{\small 1}} 
\put(140,14){\makebox(20,15)[t]{\small 2}} 
\put(155,12){\makebox(20,15)[t]{\small $\ldots$}} 
\put(165,14){\makebox(30,15)[t]{\small $k$}} 
\put(185,12){\makebox(20,15)[t]{\small $\ldots$}} 
\put(250,14){\makebox(20,15)[t]{\small 1}} 
\put(260,14){\makebox(20,15)[t]{\small 2}} 
\put(275,12){\makebox(20,15)[t]{\small $\ldots$}} 
\put(285,14){\makebox(30,15)[t]{\small $k$}} 
\put(305,12){\makebox(20,15)[t]{\small $\ldots$}} 
\put(-10,33){\makebox(20,15)[t]{\small 1}} 
\put(-10,43){\makebox(20,15)[t]{\small 2}} 
\put(-10,65){\makebox(20,15)[t]{\small $\vdots$}} 
\put(-15,73){\makebox(30,15)[t]{\small $k$}} 
\put(-10,95){\makebox(20,15)[t]{\small $\vdots$}} 
\put(110,33){\makebox(20,15)[t]{\small 1}} 
\put(110,43){\makebox(20,15)[t]{\small 2}} 
\put(110,65){\makebox(20,15)[t]{\small $\vdots$}} 
\put(105,73){\makebox(30,15)[t]{\small $k$}} 
\put(110,95){\makebox(20,15)[t]{\small $\vdots$}} 
\put(230,33){\makebox(20,15)[t]{\small 1}} 
\put(230,43){\makebox(20,15)[t]{\small 2}} 
\put(230,65){\makebox(20,15)[t]{\small $\vdots$}} 
\put(225,73){\makebox(30,15)[t]{\small $k$}} 
\put(230,95){\makebox(20,15)[t]{\small $\vdots$}} 
\put(75,14){\makebox(50,15)[t]{$x_1$}} 
\put(195,14){\makebox(50,15)[t]{$x_1$}} 
\put(315,14){\makebox(50,15)[t]{$x_1$}} 
\put(-25,115){\makebox(50,15)[t]{$x_2$}} 
\put(95,115){\makebox(50,15)[t]{$x_2$}} 
\put(215,115){\makebox(50,15)[t]{$x_2$}} 
\put(10,35){\circle*{2}} 
\put(10,45){\circle*{2}} 
\put(10,55){\circle*{2}} 
\put(10,65){\circle*{2}} 
\put(10,75){\circle*{2}} 
\put(10,85){\circle*{2}} 
\put(10,95){\circle*{2}} 
\put(10,105){\circle*{2}} 
\put(20,35){\circle*{2}} 
\put(20,45){\circle*{2}} 
\put(20,55){\circle*{2}} 
\put(20,65){\circle*{2}} 
\put(20,75){\circle*{2}} 
\put(20,85){\circle*{2}} 
\put(20,95){\circle*{2}} 
\put(20,105){\circle*{2}} 
\put(30,35){\circle*{2}} 
\put(30,45){\circle*{2}} 
\put(30,55){\circle*{2}} 
\put(30,65){\circle*{2}} 
\put(30,75){\circle*{2}} 
\put(30,85){\circle*{2}} 
\put(30,95){\circle*{2}} 
\put(30,105){\circle*{2}} 
\put(40,35){\circle*{2}} 
\put(40,45){\circle*{2}} 
\put(40,55){\circle*{2}} 
\put(40,65){\circle*{2}} 
\put(40,75){\circle*{2}} 
\put(40,85){\circle*{2}} 
\put(40,95){\circle*{2}} 
\put(40,105){\circle*{2}} 
\put(50,35){\circle*{2}} 
\put(50,45){\circle*{2}} 
\put(50,55){\circle*{2}} 
\put(50,65){\circle*{2}} 
\put(50,75){\circle*{2}} 
\put(50,85){\circle*{2}} 
\put(50,95){\circle*{2}} 
\put(50,105){\circle*{2}} 
\put(60,35){\circle*{2}} 
\put(60,45){\circle*{2}} 
\put(60,55){\circle*{2}} 
\put(60,65){\circle*{2}} 
\put(60,75){\circle*{2}} 
\put(60,85){\circle*{2}} 
\put(60,95){\circle*{2}} 
\put(60,105){\circle*{2}} 
\put(70,35){\circle*{2}} 
\put(70,45){\circle*{2}} 
\put(70,55){\circle*{2}} 
\put(70,65){\circle*{2}} 
\put(70,75){\circle*{2}} 
\put(70,85){\circle*{2}} 
\put(70,95){\circle*{2}} 
\put(70,105){\circle*{2}} 
\put(80,35){\circle*{2}} 
\put(80,45){\circle*{2}} 
\put(80,55){\circle*{2}} 
\put(80,65){\circle*{2}} 
\put(80,75){\circle*{2}} 
\put(80,85){\circle*{2}} 
\put(80,95){\circle*{2}} 
\put(80,105){\circle*{2}} 
\put(130,35){\circle*{2}} 
\put(130,45){\circle*{2}} 
\put(130,55){\circle*{2}} 
\put(130,65){\circle*{2}} 
\put(130,75){\circle*{2}} 
\put(130,85){\circle*{2}} 
\put(130,95){\circle*{2}} 
\put(130,105){\circle*{2}} 
\put(140,35){\circle*{2}} 
\put(140,45){\circle*{2}} 
\put(140,55){\circle*{2}} 
\put(140,65){\circle*{2}} 
\put(140,75){\circle*{2}} 
\put(140,85){\circle*{2}} 
\put(140,95){\circle*{2}} 
\put(140,105){\circle*{2}} 
\put(150,35){\circle*{2}} 
\put(150,45){\circle*{2}} 
\put(150,55){\circle*{2}} 
\put(150,65){\circle*{2}} 
\put(150,75){\circle*{2}} 
\put(150,85){\circle*{2}} 
\put(150,95){\circle*{2}} 
\put(150,105){\circle*{2}} 
\put(160,35){\circle*{2}} 
\put(160,45){\circle*{2}} 
\put(160,55){\circle*{2}} 
\put(160,65){\circle*{2}} 
\put(160,75){\circle*{2}} 
\put(160,85){\circle*{2}} 
\put(160,95){\circle*{2}} 
\put(160,105){\circle*{2}} 
\put(170,35){\circle*{2}} 
\put(170,45){\circle*{2}} 
\put(170,55){\circle*{2}} 
\put(170,65){\circle*{2}} 
\put(170,75){\circle*{2}} 
\put(170,85){\circle*{2}} 
\put(170,95){\circle*{2}} 
\put(170,105){\circle*{2}} 
\put(180,35){\circle*{2}} 
\put(180,45){\circle*{2}} 
\put(180,55){\circle*{2}} 
\put(180,65){\circle*{2}} 
\put(180,75){\circle*{2}} 
\put(180,85){\circle*{2}} 
\put(180,95){\circle*{2}} 
\put(180,105){\circle*{2}} 
\put(190,35){\circle*{2}} 
\put(190,45){\circle*{2}} 
\put(190,55){\circle*{2}} 
\put(190,65){\circle*{2}} 
\put(190,75){\circle*{2}} 
\put(190,85){\circle*{2}} 
\put(190,95){\circle*{2}} 
\put(190,105){\circle*{2}} 
\put(200,35){\circle*{2}} 
\put(200,45){\circle*{2}} 
\put(200,55){\circle*{2}} 
\put(200,65){\circle*{2}} 
\put(200,75){\circle*{2}} 
\put(200,85){\circle*{2}} 
\put(200,95){\circle*{2}} 
\put(200,105){\circle*{2}} 
\put(250,35){\circle*{2}} 
\put(250,45){\circle*{2}} 
\put(250,55){\circle*{2}} 
\put(250,65){\circle*{2}} 
\put(250,75){\circle*{2}} 
\put(250,85){\circle*{2}} 
\put(250,95){\circle*{2}} 
\put(250,105){\circle*{2}} 
\put(260,35){\circle*{2}} 
\put(260,45){\circle*{2}} 
\put(260,55){\circle*{2}} 
\put(260,65){\circle*{2}} 
\put(260,75){\circle*{2}} 
\put(260,85){\circle*{2}} 
\put(260,95){\circle*{2}} 
\put(260,105){\circle*{2}} 
\put(270,35){\circle*{2}} 
\put(270,45){\circle*{2}} 
\put(270,55){\circle*{2}} 
\put(270,65){\circle*{2}} 
\put(270,75){\circle*{2}} 
\put(270,85){\circle*{2}} 
\put(270,95){\circle*{2}} 
\put(270,105){\circle*{2}} 
\put(280,35){\circle*{2}} 
\put(280,45){\circle*{2}} 
\put(280,55){\circle*{2}} 
\put(280,65){\circle*{2}} 
\put(280,75){\circle*{2}} 
\put(280,85){\circle*{2}} 
\put(280,95){\circle*{2}} 
\put(280,105){\circle*{2}} 
\put(290,35){\circle*{2}} 
\put(290,45){\circle*{2}} 
\put(290,55){\circle*{2}} 
\put(290,65){\circle*{2}} 
\put(290,75){\circle*{2}} 
\put(290,85){\circle*{2}} 
\put(290,95){\circle*{2}} 
\put(290,105){\circle*{2}} 
\put(300,35){\circle*{2}} 
\put(300,45){\circle*{2}} 
\put(300,55){\circle*{2}} 
\put(300,65){\circle*{2}} 
\put(300,75){\circle*{2}} 
\put(300,85){\circle*{2}} 
\put(300,95){\circle*{2}} 
\put(300,105){\circle*{2}} 
\put(310,35){\circle*{2}} 
\put(310,45){\circle*{2}} 
\put(310,55){\circle*{2}} 
\put(310,65){\circle*{2}} 
\put(310,75){\circle*{2}} 
\put(310,85){\circle*{2}} 
\put(310,95){\circle*{2}} 
\put(310,105){\circle*{2}} 
\put(320,35){\circle*{2}} 
\put(320,45){\circle*{2}} 
\put(320,55){\circle*{2}} 
\put(320,65){\circle*{2}} 
\put(320,75){\circle*{2}} 
\put(320,85){\circle*{2}} 
\put(320,95){\circle*{2}} 
\put(320,105){\circle*{2}} 
\put(10,85){\circle{10}} 
\put(20,75){\circle{10}} 
\put(30,65){\circle{10}} 
\put(40,55){\circle{10}} 
\put(50,45){\circle{10}} 
\put(60,35){\circle{10}} 
\put(180,85){\circle{10}} 
\put(190,85){\circle{10}} 
\put(200,85){\circle{10}} 
\put(180,95){\circle{10}} 
\put(180,105){\circle{10}} 
\put(250,85){\circle{10}} 
\put(260,85){\circle{10}} 
\put(270,85){\circle{10}} 
\put(280,85){\circle{10}} 
\put(290,85){\circle{10}} 
\put(300,85){\circle{10}} 
\put(300,75){\circle{10}} 
\put(300,65){\circle{10}} 
\put(300,55){\circle{10}} 
\put(300,45){\circle{10}} 
\put(300,35){\circle{10}} 
\put(55,120){\makebox(50,15)[t]{$(a)$}} 
\put(175,120){\makebox(50,15)[t]{$(b)$}} 
\put(295,120){\makebox(50,15)[t]{$(c)$}} 
\end{picture} 
\end{center}
\vspace{-0.8cm}
\caption{Failure sets $S_k$ when failures are due to 
(a) sum of shocks, (b) minimum of shocks, and (c) maximum of shocks.}
\label{figure:1}
\end{figure}
%
\subsection{Failures due to sum of shocks}\label{section:sum}
Assume that the system fails when the sum of shocks of type 1 and of type 2 
reaches a random threshold that takes values in $\{1,2,\ldots\}$. In this case 
we have 
$$
 S_k= \{(x_1,x_2)\in \mathbb{N}^2_0: x_1+x_2=k\}, 
 \qquad k=0,1,2,\dots.
$$
and
$$ 
 \tilde S_k^{(1)}=\tilde S_k^{(2)}=S_{k-1}, 
 \qquad k=1,2,\dots.
$$
Under these assumptions, from (\ref{equation:10}) and by use of Newton's binomial 
theorem for $i=1,2$ we obtain
\begin{equation}
 f_i(t)= \lambda_i e^{-(\lambda_1+\lambda_2) t}\sum_{k=1}^{+\infty} p_k
 \frac{[(\lambda_1+\lambda_2) t]^{k-1}}{(k-1)!}.
 \label{equation:1}
\end{equation}
Similarly, making use of (\ref{equation:11}) we have 
\begin{equation}
 \barF_T(t)=  e^{-(\lambda_1+\lambda_2) t}\sum_{k=0}^{+\infty} \barP_k
 \frac{[(\lambda_1+\lambda_2) t]^k}{k!}.
 \label{equation:2}
\end{equation}
Moreover, from (\ref{equation:24}) the probability that the failure 
ultimately occurs due to a shock of type $i$ follows: 
\begin{equation}
 P(\delta=i)=\frac{\lambda_i}{\lambda_1+\lambda_2}, 
 \qquad i=1,2.
 \label{equation:3}
\end{equation}
Finally, from (\ref{equation:25}) it is not hard to prove that the conditional 
moments for $s=1,2,\ldots$ in this case are given by 
\begin{equation}
 E(T^s\,|\,\delta=i)={1\over (\lambda_1+\lambda_2)^s}\, E[(M+s-1)_s],
 \qquad i=1,2,
 \label{equation:26}
\end{equation}
where $(m)_s$ denotes the descending factorial $m(m-1)(m-2)\cdots(m-s+1)$. 
Note that the right-hand-side of (\ref{equation:26}) does not depend on 
$\delta$. This is not surprising, since the time of failure $T$ and the 
cause of failure $\delta$ are independent if the system failures are due to 
sum of shocks, as will soon be proved. It is interesting to recall that the 
relevance of independence between $T$ and $\delta$ has been pointed out by 
various authors (see, for instance, Carriere and Kochar \cite{CaKo2000}). 
\begin{proposition}\label{proposition:1}
$T$ and $\delta$ are independent under the model assumptions of 
Section \ref{section:sum}. 
\end{proposition}
\begin{proof}
From (\ref{equation:13}), (\ref{equation:1}) and (\ref{equation:3}), expressing 
$p_k$ as $\barP_{k-1}-\barP_k$ there holds:
\begin{eqnarray*}
P(\delta=i)f_T(t) \!\!\!\! 
 &=& \!\!\!\! 
\lambda_i e^{-(\lambda_1+\lambda_2)t}
\bigg\{ \sum_{k=0}^{+\infty} \sum_{j=k+1}^{+\infty} p_j
\frac{[(\lambda_1+\lambda_2) t]^k}{k!} - \sum_{k=1}^{+\infty}
\sum_{j=k+1}^{+\infty} p_j
\frac{[(\lambda_1+\lambda_2) t]^{k-1}}{(k-1)!}\bigg\} \\
 &=& \!\!\!\!
\lambda_i e^{-(\lambda_1+\lambda_2)t}
\bigg\{ \sum_{j=1}^{+\infty}  p_j \sum_{k=0}^{j-1}
\frac{[(\lambda_1+\lambda_2) t]^k}{k!} - \sum_{j=2}^{+\infty}
 p_j \sum_{k=1}^{j-1}
\frac{[(\lambda_1+\lambda_2) t]^{k-1}}{(k-1)!}\bigg\} \\
 &=& \!\!\!\!
\lambda_i e^{-(\lambda_1+\lambda_2)t}
 \sum_{j=1}^{+\infty}  p_j 
\frac{[(\lambda_1+\lambda_2) t]^{j-1}}{(j-1)!}=f_i(t), 
 \qquad i=1,2.
\end{eqnarray*}
The proof is thus completed, for $T$ and $\delta$ are independent 
if and only if $f_i(t)=P(\delta=i)f_T(t)$ for all $t\geq 0$ and $i=1,2$. 
\end{proof}
\par
For the model in which failures are due to sum of shocks, the subdensity 
$f_i(t)$ and the survival function $\barF_T(t)$ are shown in Table \ref{table:1} 
when the distribution of $M$ is (i) geometric, (ii) Poisson over 
$\{1,2,\ldots\}$, (iii) $p_k=\frac{1}{k(k+1)}$, $k=1,2,\ldots$, and  
(iv) a suitable negative binomial. $I_n(x)$ denotes the modified Bessel 
function of the first kind. 
\par
We notice that $f_i(t)$ and $\barF_T(t)$ can be expressed similarly to 
Eqs.\ (\ref{equation:1}) and (\ref{equation:2}) if $N_1(t)$ and $N_2(t)$ are 
independent non-homogeneous Poisson processes. However, in this case $P(\delta=i)$ 
would be equal to the right-hand-side of (\ref{equation:3}) only if the ratio 
of the time-varying intensity functions of $N_1(t)$ and $N_2(t)$ is a constant. 
Hence, Proposition \ref{proposition:1} still holds if $N_i(t)$ is a Poisson 
process characterized by an intensity function of the form $\lambda_i\,u(t)$, 
with $i=1,2$ and $t\geq 0$.  
\begin{table}[t]  
\footnotesize
\begin{center}
\begin{tabular}{llll}
\hline\spazio
 {} & $p_k; k\geq 1$ & $f_i(t)$ & $\barF_T(t)$ \\
\hline
\spazio
$\!\!\!\!$(i) 
&
 $p(1-p)^{k-1}$ 
&
 $p \lambda_i e^{-p\tau}$
&
 $e^{-p\tau}$
\\
\hline\spazio
$\!\!\!\!$(ii) 
&
 $\ds\frac{\eta^{k-1}e^{-\eta}}{(k-1)!}$ 
&
 $\lambda_i e^{-\eta-\tau} 
 I_0\big(2\sqrt{\eta\tau}\big)$
&
 $e^{-\eta-\tau}\ds\sum_{n=0}^{+\infty}
 \left({\eta\over \tau}\right)^{n/2} 
 I_{n}\big(2\sqrt{\eta\tau}\big)$
\\
\hline\spazio
$\!\!\!\!$(iii) 
&
 $\ds\frac{1}{k(k+1)}$ 
&
 $\ds\frac{\lambda_i}{\tau^2}\big\{1-e^{-\tau} (1+\tau)\big\}$
&
 $\ds\frac{1-e^{-\tau}}{\tau}$
\\
\hline\spazio 
$\!\!\!\!$(iv) 
&
 $kp^2(1-p)^{k-1}$ 
&
 $\lambda_i p^2 \big\{1+(1-p)\tau\big\}e^{-p\tau}$
&
 $e^{-p\tau}\big\{1+p(1-p)\tau\big\}$
\\
\hline\spazio
\end{tabular}
\end{center}
\vspace{-0.8cm}
\caption{Subdensities and survival function for the model with failures due to sum 
of shocks, with $\tau=(\lambda_1+\lambda_2)t$.}
\label{table:1}
\end{table} 
 
\par
The model considered in this paper can be easily extended to the case of an 
arbitrary number $h$ of types of shocks, by introducing a multidimensional 
counting process $(N_1(t),\ldots,N_h(t))$, where $N_i(t)$ describes the number 
of shocks of $i$-th type occurred in $[0,t]$, $i=1,2,\ldots,h$. In this case, 
under the assumption that failures are due to sum of shocks, the survival 
function of $T$ can be expressed as 
\begin{equation}
 \barF_T(t)=\sum_{k=0}^{+\infty} p_k\,P\{N_1(t)+\ldots+N_h(t)<k\}, 
 \qquad t\geq 0.
 \label{equation:22}
\end{equation}
The form of the right-hand-side of (\ref{equation:22}) suggests that a central 
limit theorem might be implemented. For instance, assuming that the components 
of $(N_1(t),\ldots,N_h(t))$ are independent and identically distributed with 
finite mean and variance, Eq.\ (\ref{equation:22}) at a fixed time $t_0>0$
for large $h$ gives
$$
  \barF_T(t_0)\approx\sum_{k=0}^{+\infty} p_k\,
 \Phi\left({k-h\,\mu_0\over \sqrt{h}\,\sigma_0}\right),
$$ 
where $\mu_0=E[N_i(t_0)]$ and $\sigma_0^2=Var[N_i(t_0)]$, with $\Phi(\cdot)$ 
denoting the standard normal distribution function. 
\subsection{Failures due to minimum of shocks}
In this Section we consider another failure scheme for the shock model 
characterized by two kinds of shocks. Namely, in this case the system fails 
when the minimum of the shocks of type 1 and of type 2 reaches a random 
threshold that takes values in $\{1,2,\ldots\}$. Such a model may for 
instance be appropriate to describe the failure of systems made out of 
units connected in parallel. 
\par
In this model the failure sets and the risky sets are respectively given by 
\begin{equation}
 S_k=\{(x_1, x_2)\in \mathbb{N}^2_0: \min (x_1,x_2)=k\}, 
 \qquad k=1,2,\dots,
 \label{equation:14}
\end{equation}
and   
$$
 \tilde S_k^{(1)}=\{(k-1, x_2)\in \mathbb{N}^2_0: 
 x_2=k, k+1, \dots\},
$$
$$
 \tilde S_k^{(2)}=\{(x_1, k-1)\in \mathbb{N}^2_0: 
 x_1=k, k+1, \dots\}.
$$
\par
From (\ref{equation:10}) and (\ref{equation:14}) the following subdensities 
now follow: 
\begin{equation}
 \begin{array}{l}
 f_1(t)=\lambda_1e^{-(\lambda_1+\lambda_2)t}
 \ds\sum_{k=1}^{+\infty} p_k \frac{(\lambda_1 t)^{k-1}}{(k-1)!}\,
 \barE_k(\lambda_2 t), \qquad t\geq 0, \\
 f_2(t)=\lambda_2 e^{-(\lambda_1+\lambda_2)t}
 \ds\sum_{k=1}^{+\infty} p_k \frac{(\lambda_2 t)^{k-1}}{(k-1)!}\,
 \barE_k(\lambda_1 t), \qquad t\geq 0, 
 \end{array} 
 \label{equation:16}
\end{equation}
where for $k=1,2,\ldots$ we have set 
$$
 \barE_k(x)=\sum_{j=k}^{+\infty}{x^j\over j!}, 
 \qquad  x\in{\mathbb R}.
$$
The survival function of $T$ can now be easily obtained from (\ref{equation:11}) 
and (\ref{equation:14}): 
\begin{equation}
\barF_T(t)= e^{-(\lambda_1+\lambda_2) t} \sum_{k=0}^{+\infty} \barP_k
\bigg\{\frac{(\lambda_2 t)^k}{k!}\,\barE_k(\lambda_1 t)
+ \frac{(\lambda_1 t)^k}{k!}\,\barE_{k+1}(\lambda_2 t)\bigg\},  
\qquad t\geq 0.
 \label{equation:19}
\end{equation}
\par
Tables \ref{table:2} and \ref{table:2bis} show respectively $f_i(t)$ and $\barF_T(t)$ 
for the present model, for three choices of the distribution of $M$. Note that the 
function 
$$
 M_{n,k}(x)=\sum_{r=0}^{+\infty}{x^{3r+n+k}\over r!(r+n)!(r+k)!}, 
$$
appearing in case (ii), is a modified two-index Bessel function (also known as a 
modified Humbert function, see for instance Dattoli {\em et al.}\ \cite{DaLoMaToVoCh94}). 
It appears also in the probability distributions of certain two-dimensional 
random motions (see \cite{CeDiCr01} and \cite{DiCr2002}). The survival function 
in case (ii) of Table \ref{table:2bis} has been obtained by recalling the 
relation (see, for instance, Eq.\ (5.3) of \cite{CeDiCr01}) 
$$
 \sum_{n=0}^{+\infty}\sum_{m=0}^{+\infty}\alpha^n\beta^m\,M_{n,m}(x)
 =\exp\left\{x\left(\alpha+\beta+{1\over \alpha\beta}\right)\right\}. 
$$
\begin{table}[t]  
\footnotesize
\begin{center}
\begin{tabular}{ll}
\hline\spazio
{} & $f_i(t)$ \\
\hline
\spazio
$\!\!\!\!$(i) 
& 
 $p\lambda_i e^{-\tau}
 \ds\sum_{n=0}^{+\infty} \left[{\lambda_{3-i}\over (1-p)\lambda_i}\right]^{n+1 \over 2} 
 I_{n+1}(\alpha t)$
\\
\hline\spazio
$\!\!\!\!$(ii) 
& 
 $\lambda_i e^{-\eta-\tau}
 \ds\sum_{n=0}^{+\infty} \left({\beta^2\over \alpha\lambda_i t}\right)^{n+1} 
 M_{0,n+1}(\beta)$
\\
\hline\spazio
$\!\!\!\!$(iii) 
&
 $\ds{e^{-\tau}\over \lambda_i t^2}\left\{\sqrt{\lambda_i\over \lambda_{3-i}}\, 
 I_1\left(2t\sqrt{\lambda_1\lambda_2}\right)-\lambda_i t
 +\ds\sum_{r=0}^{+\infty}\left[\sqrt{\lambda_{3-i}\over \lambda_i}
 I_r\left(2t\sqrt{\lambda_1\lambda_2}\right)
 -\ds{(\lambda_{3-i} t)^r\over r!}-{\lambda_i t(\lambda_{3-i} t)^{r+1}\over (r+1)!}\right]\right\} $
\\
\hline\spazio 
\end{tabular}
\end{center}
\vspace{-0.8cm}
\caption{Subdensities for the model with failures due to minimum of shocks, 
for cases (i)-(iii) of Table 1, with $\tau=(\lambda_1+\lambda_2)t$, 
$\alpha=2\sqrt{(1-p)\lambda_1\lambda_2}$ and $\beta=(\eta \lambda_1\lambda_2 t^2)^{1/3}$.}
\label{table:2}
\end{table} 

\begin{table}[t]  
\footnotesize
\begin{center}
\begin{tabular}{lll}
\hline\spazio
{} & $\barF_T(t)$ \\
\hline
\spazio
$\!\!\!\!$(i) 
& 
 $e^{-\tau} \ds\sum_{n=0}^{+\infty} 
 \bigg\{\left[{\lambda_{1}\over (1-p)\lambda_2}\right]^{n\over 2} 
 I_{n}(\alpha t)+\left[\ds{\lambda_{2}\over (1-p)\lambda_1}\right]^{n+1\over 2} 
 I_{n+1}(\alpha t)\bigg\}$
\\
\hline\spazio
$\!\!\!\!$(ii) 
& 
 $e^{-\eta-\tau}\left(\exp\left\{\beta\left[\left(\ds{\eta^2\over \lambda_1\lambda_2 t^2}\right)^{\!\! {1\over 3}}
 +\left(\ds{\lambda_1^2 t^2\over \eta}\right)^{\!\! {1\over 3}}
 + \left(\ds{\lambda_2 \over \lambda_1\eta^2}\right)^{\!\! {1\over 3}}\right]\right\}
+\exp\left\{\beta\left[\left(\ds{\eta^2\over \lambda_1\lambda_2 t^2}\right)^{\!\! {1\over 3}}
 +\left(\ds{\lambda_2^2 t\over \lambda_1\eta}\right)^{\!\! {1\over 3}}+
 \right.\right.\right.\hspace{-0.5cm}$
\\
{}
&
 $\left.\left.\left. + \left(\ds{\lambda_1^2 t \over \lambda_2\eta}\right)^{\!\! {1\over 3}}\right]\right\}
 -\left(\ds{\lambda_2^2 t \over \lambda_1\eta}\right)^{\!\! {1\over 3}}
 \ds\sum_{j=0}^{+\infty}M_{j,0}(\beta)\left(\ds{\eta^2\over \lambda_1\lambda_2 t^2}\right)^{\!\! {j\over 3}}
 \right)$
\\
\hline\spazio
$\!\!\!\!$(iii) 
&
 $\ds{e^{-\tau}\over t}\left\{\ds\sum_{j=0}^{+\infty}
 I_j\left(2t\sqrt{\lambda_1\lambda_2}\right)
 {\lambda_1^{j+1}+\lambda_2^{j+1}\over (\lambda_1\lambda_2)^{{j\over 2}+1}}
 +I_1\left(2t\sqrt{\lambda_1\lambda_2}\right){1\over \sqrt{\lambda_1\lambda_2}}
 -{e^{\lambda_1 t}\over \lambda_2}-{e^{\lambda_2 t}\over \lambda_1}\right\}$
& 
{}\\
\hline\spazio 
\end{tabular}
\end{center}
\vspace{-0.8cm}
\caption{Survival function for the model with failures due to minimum 
of shocks, for the same cases of Table 2.}
\label{table:2bis}
\end{table} 
%
\subsection{Failures due to maximum of shocks}
Let us now consider a failure scheme in which the system fails when the 
maximum of the shocks of type 1 and of type 2 reaches a random threshold   
taking values in $\{1,2,\ldots\}$. For instance, this model is suitable to 
describe the failure of systems composed by units serially interconnected. 
\par
In this case we assume that the failure sets are given by 
\begin{equation}
 S_k=\{(x_1, x_2)\in \mathbb{N}^2_0: 
 \max (x_1,x_2)=k\}, \qquad  k=1,2,\dots,
 \label{equation:15}
\end{equation}
so that, for $k=1,2,\dots$, the risky sets are 
$$
 \tilde S_k^{(1)}=\{(k-1, x_2)\in \mathbb{N}^2_0: 
 x_2=0, 1,\ldots k-1\},
$$
$$ 
 \tilde S_k^{(2)}=\{(x_1, k-1)\in \mathbb{N}^2_0: 
 x_1=0, 1, \dots k-1\}.
$$
Note that, even if definition (\ref{equation:15}) includes state $(k,k)$ in $S_k$, 
the assumption that failures are due to a single cause excludes in this model 
the possibility that states of the type $(k,k)$ are failure states.  
\par
From (\ref{equation:10}) it follows 
\begin{equation}
 \begin{array}{l}
 f_1(t)=\lambda_1 e^{-(\lambda_1+\lambda_2)t}
 \ds\sum_{k=1}^{+\infty} p_k \frac{(\lambda_1 t)^{k-1}}{(k-1)!}E_{k}(\lambda_2 t), 
 \qquad t\geq 0, \\
 f_2(t)=\lambda_2 e^{-(\lambda_1+\lambda_2)t}
 \ds\sum_{k=1}^{+\infty} p_k \frac{(\lambda_2 t)^{k-1}}{(k-1)!}E_{k}(\lambda_1 t), 
 \qquad t\geq 0,
 \end{array}
 \label{equation:17}
\end{equation}
where $E_k(x)$, $x\in{\mathbb R}$, is the auxiliary function such that $E_0(x)=0$ and 
\begin{equation}
 E_k(x)=\sum_{j=0}^{k-1}{x^j\over j!} 
 =e^{x}-\barE_k(x), 
 \qquad 
 k=1,2,\ldots.
 \label{equation:18}
\end{equation}
It is not hard to see that the subdensities $f_i(t)$ for the two models in which 
failures are due to minimum and maximum of shocks are closely related. This is 
due to the similar nature of failure sets $S_k$ defined in (\ref{equation:14}) 
and (\ref{equation:15}). Indeed, making use of identity (\ref{equation:18}), the 
sum of subdensities (\ref{equation:16}) and (\ref{equation:17}) is seen to be 
$$
 \lambda_i e^{-\lambda_i t} \ds\sum_{k=1}^{+\infty} p_k \frac{(\lambda_i t)^{k-1}}{(k-1)!},  
 \qquad t\geq 0,\quad i=1,2.
$$
Moreover, from (\ref{equation:11}) and (\ref{equation:15})  
the survival function of $T$ is obtained: 
\begin{equation}
\barF_T(t)= e^{-(\lambda_1+\lambda_2) t} \sum_{k=0}^{+\infty} \barP_k
\bigg\{\frac{(\lambda_2 t)^k}{k!}\,E_{k}(\lambda_1 t)
+ \frac{(\lambda_1 t)^k}{k!}\,E_{k+1}(\lambda_2 t)\bigg\},  
\qquad t\geq 0.
 \label{equation:20}
\end{equation}
As for subdensities $f_i(t)$, the survival functions of $T$ for the last two 
considered models are related. Indeed, making again use of (\ref{equation:18}) 
the sum of functions (\ref{equation:19}) and (\ref{equation:20}) is
$$
 \sum_{i=1}^2 e^{-\lambda_i t}\ds\sum_{k=0}^{+\infty}\barP_k\frac{(\lambda_i t)^k}{k!},  
 \qquad t\geq 0.
$$
\section{Concluding remarks}
We have introduced a formulation of shock models driven by bivariate Poisson 
process that includes a competing risks set-up, in which failures are due to 
one out of two mutually exclusive causes. We have obtained specific expressions 
for the failure densities and the survival function under three different failure 
schemes, in which the variable that causes the failure is (a) the sum of shocks, 
(b) the minimum of shocks, (c) the maximum of shocks. 
\par
Future developments of this model will be oriented to the construction of 
other structures for failure sets $S_k$, and to the study of a case in which the 
two fonts of shocks are not mutually exclusive, thus including the possibility 
that failures are due to both types of shocks. The underlying counting process 
will not have independent components; for instance the following bivariate 
Poisson process could be employed (see Marshall and Olkin \cite{MaOl67} and 
references therein):
$$
 P\{N_1(t)=x_1,N_2(t)=x_2\}
 =e^{-(\lambda_1 + \lambda_2 + \lambda_3) t}\sum_{k=0}^{\min(x_1,x_2)}
 \frac{\lambda_1^{x_1-k} \,\lambda_2^{x_2-k}\, \lambda_3^{k}\, t^{x_1+x_2-k}}
 {(x_1-k)!\,(x_2-k)!\,k!}, 
 \qquad t\geq 0.
$$
Furthermore, the structure of the failure sets and of the risky sets will be 
modified accordingly, by taking also into account the necessity that a third 
hazard rate must be added to those defined in Eqs.\ (\ref{equation:7}). 
\subsection*{\bf Acknowledgments}
%
The authors thank Petr L\'ansk\'y, Franco Pellerey, Luigi M.\ Ricciardi and 
Hiromi Seno for useful comments and fruitful discussions. This work has been 
performed under partial support by MIUR (PRIN 2005), G.N.C.S.-INdAM and 
Regione Campania.
%

%
{\it Antonio Di Crescenzo}\\ 
{\sc Dipartimento di Matematica e Informatica, Universit\`a di Sa\-ler\-no} \\ 
Via Ponte don Melillo, 84084 Fisciano (SA), Italy\\
E-mail: {\tt adicrescenzo@unisa.it}\\
\newline
{\it Maria Longobardi} \\ 
{\sc Dipartimento di Matematica e Applicazioni, Universit\`a di Napoli Federico II} \\ 
Via Cintia, 80126 Napoli, Italy\\
E-mail: {\tt maria.longobardi@unina.it}
%
 
\end{document}